\documentclass[sn-mathphys]{sn-jnl}% Math and Physical Sciences Reference Style
\jyear{2021}%

\theoremstyle{thmstyleone}% first theorem style
\newtheorem{theorem}{Theorem}%  meant for continuous numbers
\newtheorem{proposition}[theorem]{Proposition}% 
\newtheorem{lemma}{Lemma}%
\newtheorem{corollary}{Corollary}%

\newtheorem{conditional proposition}{Conditional Proposition}

\theoremstyle{thmstyletwo}% second theorem style
\newtheorem{remark}{Remark}%

\theoremstyle{thmstylethree}% third theorem style

\usepackage{graphicx}
\usepackage{soul}
\usepackage[symbol]{footmisc}
\usepackage{xcolor}
\usepackage{mathtools}
\usepackage{amsmath}
\usepackage{cases}

\newcommand{\zero}{\mathbf{0}}
\newcommand{\one}{\mathbf{1}}
\newcommand{\abs}[1]{\lvert#1\rvert}

\let\oldprec\prec
\renewcommand{\prec}{\oldprec\hspace{-2pt}}

\let\oldsucc\succ
\renewcommand{\succ}{\hspace{-2pt}\oldsucc}

\newcommand{\edit}[1]{#1}

\begin{document}

\title[An extension to ``A subsemigroup of the rook monoid'']{
An extension to ``A subsemigroup of the rook monoid''}

\author*[1]{George Fikioris}\email{gfiki@ece.ntua.gr}
\author[2]{Giannis Fikioris}\email{gfikioris@cs.cornell.edu}

\affil*[1]{\orgdiv{School of Electrical and Computer Engineering}, \orgname{National Technical University of Athens}, \orgaddress{\street{Zografou}, \city{Athens}, \postcode{GR 15773}, \state{Attica}, \country{Greece}}}

\affil[2]{\orgdiv{Department of Computer Science}, \orgname{Cornell University}, \orgaddress{\street{Hoy Rd}, \city{Ithaca}, \postcode{NY 14853}, \state{New York}, \country{USA}}}

\abstract{
\unboldmath{}
A recent paper studied an inverse submonoid $M_n$ of the rook monoid, by representing the nonzero elements of $M_n$ via certain triplets belonging to $\mathbb{Z}^3$. In this short note,  we allow the triplets to belong to $\mathbb{R}^3$. We thus study a new inverse monoid $\overline{M}_n$, which is a supermonoid of $M_n$. We point out similarities and find essential differences. We show that $\overline{M}_n$ is a noncommutative, periodic, combinatorial, fundamental, completely semisimple, and strongly $E^*$-unitary inverse monoid.
}


\maketitle


\bmhead{Acknowledgments}
We thank the reviewer, whose comments and suggestions notably improved this work.
The work of Giannis Fikioris was supported in part by the Department of Defense (DoD) through the National Defense Science \& Engineering Graduate (NDSEG) Fellowship Program, the Onassis Foundation -- Scholarship ID: F ZS 068-1/2022-2023, and AFOSR grant FA9550-23-1-0068.

\section{Introduction}

The symmetric inverse semigroup $\mathcal{IS}_n$, also known as the rook monoid, consists of the partial injective transformations of $\{1,2,\ldots,n\}$ \cite{Ganyushkin,solomon}. Any element of $\mathcal{IS}_n$ can be represented as an $n\times n$ matrix whose entries are $0$ or $1$, with at most one $1$ in every row and every column.

In a previous paper \cite{fikioris-fikioris}, we introduced a submonoid $M_n$ of $\mathcal{IS}_n$ and studied its properties. The monoid $M_n$ consists of the zero matrix together with those matrices of $\mathcal{IS}_n$ whose $1$s lie on a single diagonal and form an uninterrupted block (i.e., no $0$ lies between any two $1$s). Let $d$ be the said diagonal ($d=-n+1,\ldots, n-1$, with $d=0$ being the main diagonal), let $k$ be the row of the northwestern $1$, and let $m$ be the row of the southeastern $1$. The study of \cite{fikioris-fikioris} was facilitated by representing the elements of $M_n$ as triplets $\langle d,k,m\rangle \in\mathbb{Z}^3$ ($d$, $k$, and $m$ are appropriately restricted), and developing a closed-form expression representing the product of two elements.

This short note is an extension that allows $\langle d,k,m\rangle \in\mathbb{R}^3$; the restrictions on the parameters $d$, $k$, $m$, as well as the product formula, remain unaltered. We thus study a new monoid $\overline{M}_n$, of which the
$M_n$ of \cite{fikioris-fikioris} is a submonoid. For reasons of symmetry, we  switch the order of the first two arguments and use the notation $\prec k,d, m\succ$ for an $x\in\overline{M}_n$, so that
\begin{equation}
\label{eq:notationchange}
  x=\prec k,d, m \succ=\langle d,k,m \rangle,\quad k,d,m\in\mathbb{Z}.
\end{equation}
To facilitate comparisons with ``the integer case,'' however, we maintain much of the notation of \cite{fikioris-fikioris}. For example, we retain the symbol $x^T$ for the semigroup inverse; the underlying reason is that inverting $x\in M_n$ amounts to transposing the matrix represented by $x$.  As in \cite{fikioris-fikioris}, $\zero$ and $\one$ denote monoid zero and identity, and ideal means two-sided ideal. We use the traditional notations
for Green’s relations, associated equivalence classes, and principal ideals, as well as the usual notations $\mathbf{r}(x)$ and $\mathbf{d}(x)$ \cite{Lawson,lawson2023introduction} for $xx^T$ and $x^Tx$, respectively. A $j$th root of $x\in \overline{M}_n$ is a $y\in \overline{M}_n$ such that $x= y^j$ ($j \in \mathbb{N}$).

\section{The inverse monoid \texorpdfstring{$\overline{M}_n$}{Mn}}
\label{section:definitions}

Let $n\in\mathbb{Z}$ with $n\ge 2$. Our definition of $\overline{M}_n$ is
\begin{equation}
\label{eq:m-definition}
\begin{split}
\overline{M}_n=\{\zero\}
\cup
\{\prec k, & d, m \succ:\  k,d,m\in \mathbb{R};\\ &1-\min(0,d)\le k \le m\le n-\max(0,d)\}. \end{split}
\end{equation}
Note that the restrictions in (\ref{eq:m-definition}) further imply
\begin{equation}
\label{eq:restrictions-further}
    -(n-1)\le d\le n-1\quad \mathrm{and}\quad  1\le k\le m\le n.
\end{equation}
As in \cite{fikioris-fikioris}, the formula for the product of two nonzero elements is
\begin{equation}
\label{eq:multiplication-finite}
\prec k, d, m \succ \prec k', d', m' \succ=
\begin{cases}
\prec k'', d'', m'' \succ,\quad k'' \le m'',\\\
\zero,\quad k''> m'',
\end{cases}
\end{equation}
in which the parameters $k''$, $d''$, and $m''$ are 
\begin{equation}
\label{eq:doubleprime}
k''=\max(k,k'-d),\quad
d''=d+d',\quad
m''=\min(m,m'-d).
\end{equation}
We can use the definitions (\ref{eq:m-definition}), (\ref{eq:multiplication-finite}), and (\ref{eq:doubleprime}) to show that $\overline{M}_n$ is a monoid with $\one=\prec 1, 0, n \succ$. We can also verify a formula for powers:
\begin{lemma}
\label{lemma:powers}
For $x=\prec k, d, m \succ\in \overline{M}_n\setminus\{\zero\}$ and $j\in\mathbb{N}$ we have
\begin{equation}
\label{eq:powers-1}
x^j=\begin{cases}
\prec k^{(j)}, d^{(j)}, m^{(j)}\succ,\quad \textrm{if}\quad  k^{(j)}\le m^{(j)},
\\
\zero,\quad \textrm{if}\quad  k^{(j)}>m^{(j)},
\end{cases}
\end{equation}
where
\begin{equation}
\label{eq:powers-2}
k^{(j)}=k-(j-1)\min(0,d),\quad 
d^{(j)}=jd,\quad 
m^{(j)}=m-(j-1)\max(0,d).
\end{equation}
In particular, $x^2=x$ iff $d=0$. 
\end{lemma}
Let us define 
\begin{equation}
\label{eq:transpose}
x^T=
\begin{cases}
\zero,\quad x=\zero,\\
\prec k+d, -d, m+d \succ, \quad x=\prec k, d, m \succ\in \overline{M}_n \setminus\{\zero\}.
\end{cases}
\end{equation}
The inequalities in (\ref{eq:m-definition}) ensure that $x^T\in \overline{M}_n$. The multiplication formula (\ref{eq:multiplication-finite}) gives
\begin{equation}
\label{eq:productxtimesinversex}
\mathbf{r}(x)=xx^T=
\begin{cases}
\zero,\quad x=\zero,\\
\prec k, 0, m \succ, \quad x=\prec k, d, m \succ\in M_n\setminus\{\zero\}
\end{cases}
\end{equation}
and
\begin{equation}
\label{eq:productinversextimesx}
\mathbf{d}(x)=x^Tx=
\begin{cases}
\zero,\quad x=\zero,\\
\prec k+d, 0, m+d \succ, \quad x=\prec k, d, m \succ\in M_n\setminus\{\zero\}, 
\end{cases}
\end{equation}
as well as $xx^Tx=x$ and $x^Txx^T=x^T$. Thus $x^T$ is an inverse of $x$ and $\overline{M}_n$ is a regular semigroup. The idempotents are $\zero$ and the elements $\prec k, 0, m \succ$; and by (\ref{eq:multiplication-finite}), these idempotents commute. Accordingly \cite{Lawson}, $\overline{M}_n$ is an inverse semigroup. In sum, we have arrived at  
\begin{proposition}
\label{th:basic-theorem}
The $\overline{M}_n$ defined in (\ref{eq:m-definition})--(\ref{eq:doubleprime}) is a 
noncommutative inverse monoid with zero, whose identity is $\one=\prec 1, 0, n \succ$. The $x^T$ given in (\ref{eq:transpose}) is the unique inverse  of $x\in \overline{M}_n$. The semilattice  of idempotents---to be denoted by $\mathbf{E}\left(\overline{M}_n\right)$---consists of $\zero$ together with all elements $\prec k, 0, m \succ$ (in which $d=0$). 
\end{proposition}

Inverse semigroups are associated with a natural partial order \cite{Lawson,lawson2023introduction,Howie} which, for our nonzero elements, can be  formulated in terms of  triplet parameters:

\begin{corollary}
\label{corollary:natural-partial-order}
Let $\le$ be the natural partial order in $\overline{M}_n\setminus\{\zero\}$. Then
\begin{equation}
\label{eq:natural-partial-order}
    \prec k, d, m \succ \,\le\, \prec k', d', m' \succ\iff d=d',\  k\ge k',\ \mathrm{and}\ m\le m'. 
\end{equation}
\end{corollary}
\begin{proof}
As $x\le y$ iff $x=xx^Ty$ \cite{Lawson,lawson2023introduction,Howie}, the assertion follows easily from (\ref{eq:multiplication-finite}) and (\ref{eq:productxtimesinversex}).
\end{proof}

\begin{corollary}
\label{corollary:e*unitary}
    The inverse semigroup $\overline{M}_n$ is $E^*$-unitary (also called $0$-$E$ unitary).
\end{corollary}
\begin{proof}
If $\zero\ne x\le y$ and $x$ is idempotent, then $y\ne \zero$, so we can set $x=\prec k, 0, m \succ$ and $y=\prec k', d', m' \succ$. Then $d'=0$ by (\ref{eq:natural-partial-order}), so that $y$ is idempotent. Therefore $\overline{M}_n$ is $E^*$-unitary by definition \cite{Lawson,lawson2023introduction}.
\end{proof}
Let $C_n$ denote the set of all closed real intervals $[k,m]$ for which $1\le k\le m\le n$. Thus $C_n$ consists of the closed line segments within $[1,n]$.

\begin{corollary}
\label{corollary:semilattice} The semilattice of idempotents
$\mathbf{E}\left(\overline{M}_n\right)$ is isomorphic to $C_n$, with multiplication corresponding to segment intersection.  
\end{corollary}
\begin{proof}
Taking $d=d'=0$ in (\ref{eq:m-definition})--(\ref{eq:doubleprime}) gives $k''=\max(k,k')$, $m''=\min(m,m')$, and
\begin{equation*}
\prec k, 0, m \succ \prec k', 0, m' \succ=
\begin{cases}
\prec k'', 0, m''\succ,\quad k''\le  m'',\\
\zero,\quad k''>m'',
\end{cases}
\end{equation*}
for $1\le k \le m \le n$ and $1\le k' \le m' \le n$. This proves the assertion.
\end{proof}

Any inverse semigroup gives rise to an underlying groupoid, within which a restricted product $x\cdot y$ is defined \cite{Lawson,lawson2023introduction}. In $\overline{M}_n$, the underlying groupoid is readily described using triplet parameters:
\begin{corollary} 
\label{corollary:underlying-groupoid}
Let $x,y\in \overline{M}_n$. If $x=\zero$, then $x\cdot y$ is defined iff $y=\zero$.  And if $\zero\ne x=\prec k, d, m \succ$, then  $x\cdot y$ is defined iff $\zero\ne y=\prec k', d', m' \succ$, with
    \begin{equation}
    \label{eq:y-restricted-product}
    k'=k+d\mathrm{\ and\ }m'=m+d,
    \end{equation}
    in which case
    \begin{equation}
        x\cdot y=xy=\prec k, d+d', m \succ.
    \end{equation}
\end{corollary}
\begin{proof}
In an inverse semigroup, $x\cdot y$ is defined iff $\mathbf{d}(x)=\mathbf{r}(y)$, in which case $x\cdot y=xy$ \cite{Lawson,lawson2023introduction}. The assertions then follow from (\ref{eq:multiplication-finite}), (\ref{eq:doubleprime}), (\ref{eq:productxtimesinversex}), and 
(\ref{eq:productinversextimesx}).
\end{proof}
\section{Connections with \texorpdfstring{${M}_n$}{Mn}; nilpotents; graphical interpretations;  the height function}
\label{section:graphical}

We obtain the submonoid $M_n$ of \cite{fikioris-fikioris} if, in (\ref{eq:m-definition}), we replace the condition $k,d,m\in\mathbb{R}$ by the more restrictive one   $k,d,m\in\mathbb{Z}$. In $M_n$, the triplet of integers represents the $n\times n$ matrix described in \cite{fikioris-fikioris} and our Introduction. Analogously, we can interpret the triplets of $\overline{M}_n$ as line segments that are contained within a $(n-1)\times (n-1)$ square and are parallel to the diagonal shown in Fig.~\ref{fig:square}. Note that segment endpoints are permitted to lie on the square boundary, including its corners. 

\begin{remark}
    It goes without saying that a segment representing a nonzero product remains within the closed square. In particular, it is true that
\begin{equation}
\label{eq:condition-for-nonzero}
    \prec k, d, m \succ \prec k', d', m' \succ\ne \zero\implies  \abs{d+d'}\le n-1.
\end{equation}
(\ref{eq:condition-for-nonzero}), which can readily be verified directly from (\ref{eq:m-definition})--(\ref{eq:doubleprime}), will prove useful in Section~\ref{section:additional}. Stated otherwise, the inequality $\abs{d+d'}>n-1$ guarantees that the product is zero.
\end{remark}

Much of Section~\ref{section:definitions} can now be interpreted graphically: $\zero$ corresponds to the square being empty,  while $\one$ corresponds to the diagonal shown in Fig.~\ref{fig:square}. Two inverse segments ($x$ and $x^T$) are symmetric with respect to this diagonal. The segments that lie \textit{on} the diagonal make up the nonzero idempotents,  with multiplication in  $\mathbf{E}\left(\overline{M}_n\right)$ corresponding to segment intersection, as found in Corollary~\ref{corollary:semilattice}. 

The segment $\mathbf{r}(x)$ ($\mathbf{d}(x)$) is obtained by horizontally (vertically) translating $x$ till the diagonal. It follows that the segments $x$ and $y$ are horizontal (vertical) translations of each other iff $\mathbf{r}(x)=\mathbf{r}(y)$ ($\mathbf{d}(x)=\mathbf{d}(y)$), a fact to be used in Section~\ref{section:additional}.

Two segments of $\overline{M}_n$ are comparable---in the sense of the natural partial order discussed in Corollary~\ref{corollary:natural-partial-order}---iff one lies upon and is contained within the other, in which case the shorter segment is $\le$ the longer one.  Corollary~\ref{corollary:underlying-groupoid} means that the restricted product $x\cdot y$ of the underlying groupoid is defined iff $y$ is a horizontal translation of the inverse $x^T$ 
(equivalently, iff $x$ is a vertical translation of the inverse $y^T$); and that, when defined, $x\cdot y$ is a horizontal translation of $x$.

\begin{figure}
    \centering
    \includegraphics[width=.5\textwidth]{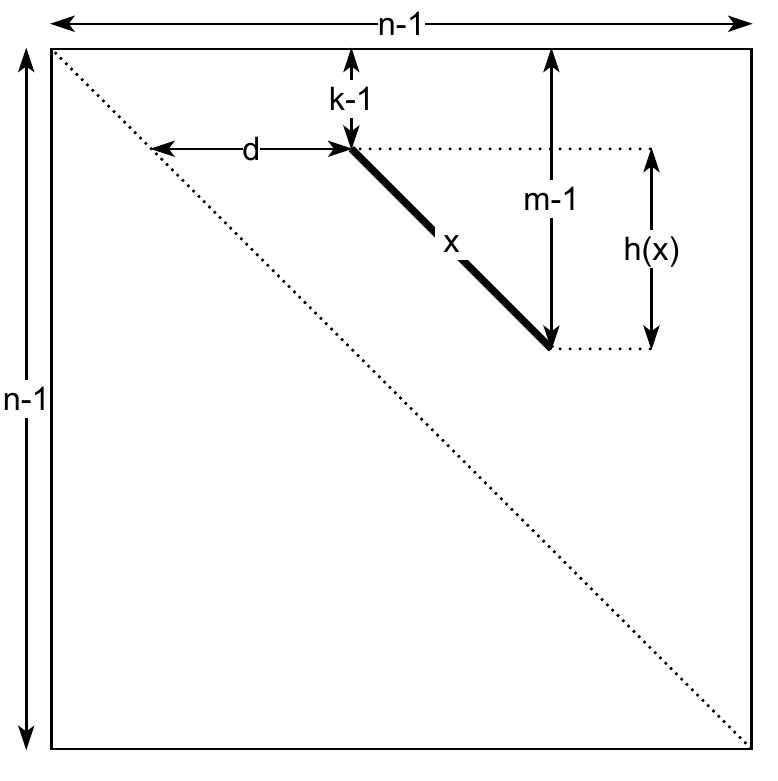}
    \caption{Any triplet $\prec k, d, m \succ$ of $\overline{M}_n\setminus \{\zero\}$ corresponds to a line segment akin to the depicted $x$, whose height is $h(x)$. The element $\zero$ corresponds to the empty square, with $h(\zero)=-1$.}
    \label{fig:square}
\end{figure}

Since (\ref{eq:m-definition}) allows $m=k$, our line segments can reduce to points within the aforementioned closed square. We use $\overline{P}_n$ to denote the set of points, viz., 
\begin{equation}
\label{eq:points}
    \overline{P}_n=\{\prec k, d, m \succ\,\in \overline{M}_n\setminus\{\zero\}:\ \  m=k\}.
\end{equation}

Let $h(x)$ denote the height of the segment $x\in \overline{M}_n$ (see Fig.~\ref{fig:square}), so that
\begin{equation}
\label{eq:v}
h(x)=
\begin{cases}
-1,\quad x=\zero,\\
m-k, \quad x=\prec k, d, m \succ\,\in \overline{M}_n\setminus\{\zero\}.
\end{cases}
\end{equation}
The height arises in a natural manner throughout; we start with some simple properties and applications.  By (\ref{eq:m-definition}) and  (\ref{eq:v}), 
\begin{equation}
\label{eq:rank-range}
    0\le h(x)\le n-\abs{d}-1\le n-1,\quad x\ne\zero,
\end{equation}
while $h$ assumes the particular values $-1$, $0$, and $n-1$ according to 
\begin{equation}
\label{eq:rank-special-cases}
    h(x)=-1\Leftrightarrow x=\zero,
        \quad h(x)=0\Leftrightarrow x\in \overline{P}_n,
        \quad h(x)=n-1\Leftrightarrow x=\one.
\end{equation}

Horizontal/vertical translations maintain the height. In other words, 
\begin{equation}
\label{eq:height-maintained}
    h(x)=h\left(\mathbf{r}(x)\right)=h\left(\mathbf{d}(x)\right).
\end{equation}
Eqns. (\ref{eq:height-maintained}) follow from (\ref{eq:productxtimesinversex}),(\ref{eq:productinversextimesx}),  and (\ref{eq:v}).

We now turn to the height of products. It follows from (\ref{eq:multiplication-finite}), (\ref{eq:doubleprime}), and (\ref{eq:v}) that $h(xy)\le \min\left(h(x),h(y)\right)$. By induction, we then get
\begin{equation}
\label{eq:rank-property}
    h(x_1x_2\ldots x_j)\le h(x_i),\mathrm{\ for\ all\ } i\in \{1,2,\ldots,j\},\quad x_i\in \overline{M}_n.
\end{equation}

\begin{remark}
\label{remark:rnk}
Ref.~\cite{fikioris-fikioris} uses 
the symbol $\mathrm{rnk}(x)$ for the rank of the partial transformation represented by $x\in M_n$. Thus in the integer case we have
\begin{equation}
\label{eq:rnk}
    \mathrm{rnk}(x)=h(x)+1,\quad x\in M_n,
\end{equation}
which 
shows why we chose the seemingly arbitrary value $h(\zero)=-1$ in (\ref{eq:v}). 
\end{remark}

If $\one=xy$, then $n-1=h(\one)\le h(x)$ by (\ref{eq:rank-special-cases}) and (\ref{eq:rank-property}), so that $x=\one$ by (\ref{eq:rank-range})  and (\ref{eq:rank-special-cases}). Similarly, $\one=xy$  implies $y=\one$. We have thus shown that 
\begin{equation}
\label{eq:sn-semigroup-2}
   xy=\one\implies x=y=\one, \quad x,y\in\overline{M}_n.
\end{equation}
Therefore $\overline{M}_n$ is actually a semigroup with a $\one$ adjoined. We will denote the inverse semigroup $\overline{M}_n\setminus\{\one\}$ by $\overline{S}_n$. 

The result that follows has no counterpart in the integer case. By means of an affine transformation $\varphi$ (easily visualized by means of Fig.~\ref{fig:square}), we demonstrate that all $\overline{M}_n$ are isomorphic: 

\begin{proposition}
\label{prop:isomorphism}
Let $n,q$ be integers $\ge 2$. The map $\varphi: \overline{M}_n\rightarrow \overline{M}_q$ given by 
\begin{equation}
    \zero\mapsto \zero, \quad x_n=\prec k_n, d_n, m_n \succ\,\mapsto\, x_q=\prec k_q, d_q, m_q \succ,
\end{equation}
where
\begin{equation}
    k_q-1=\frac{q-1}{n-1} (k_n-1),\quad
    d_q=\frac{q-1}{n-1} d_n,\quad  m_q-1=\frac{q-1}{n-1} (m_n-1),
\end{equation}
is a monoid isomorphism. Therefore any  $\overline{M}_n$  is isomorphic to $\overline{M}_2$, and any $\overline{S}_n$ is isomorphic to $\overline{S}_2$.
\end{proposition}
\begin{proof}
$\varphi$ is bijective by (\ref{eq:m-definition}). $\varphi(x_ny_n)=\varphi(x_n)\varphi(y_n)$ and $\varphi(\prec1, 0, n \succ)=\prec~1, 0, q \succ$ follow from (\ref{eq:multiplication-finite}).
\end{proof}

\begin{remark}
\label{remark:no-n}
By Proposition~\ref{prop:isomorphism}, a stand-alone study of $\overline{M}_n$ would be facilitated if one took $n=2$, corresponding to segments lying in a $1\times 1$ closed square.  However, we retain the parameter $n$ in order to draw upon and compare to results from \cite{fikioris-fikioris}. 
\end{remark}

\begin{remark}
\label{remark:h-instead-of-rnk}
Eqn. (\ref{eq:v}) and  Proposition~\ref{prop:isomorphism} imply that, for $x_n \in \overline{M}_n\setminus\{\zero\}$,
\begin{equation}
\label{eq:v-isomorphism}
    h(x_q)=\frac{q-1}{n-1} h(x_n).
\end{equation}
Eqn. (\ref{eq:v-isomorphism}) shows why, for $\overline{M}_n$, we use the height $h(x)$ instead of extending (to the non-integer case) the quantity $\mathrm{rnk}(x)=m-k+1$ mentioned in Remark~\ref{remark:rnk}: In $\overline{M}_n$, the latter quantity would scale in an unnatural manner.
\end{remark}

The idempotents of $\overline{M}_n$ were identified in Proposition~\ref{th:basic-theorem}.
The next proposition states that all other $x\in \overline{M}_n$ are nilpotents, and gives the nilpotent indexes $i(x)$. In contrast to the integer case of $M_n$ (and as expected from the aforementioned isomorphism, which leaves $i(x)$ unaltered), $i(x)$ can take on values larger than $n$.

\begin{proposition}
\label{prop:idempotent-nilpotent}
An element $x=\prec k, d, m \succ\,\in \overline{M}_n\setminus \{\zero\}$ is nilpotent if $d\ne 0$. The index $i(x)$ of the nilpotent is given by
\begin{equation}
\label{eq:index-of-nilpotent}
i(x)
=2+\left\lfloor \frac{m-k}{\abs d}
\right\rfloor=2+\left\lfloor \frac{h(x)}{\abs d}
\right\rfloor,
\end{equation}
where $\lfloor \beta \rfloor$ denotes the floor of $\beta\in\mathbb{R}$. In particular, $i(x)=2$ when $x\in \overline{P}_n$; and $i(x)\to\infty $ as $d\to 0$ (with $h(x)=m-k$ held fixed and positive).
\end{proposition}
\begin{proof}
When $d>0$, the $m^{(j)}$ in (\ref{eq:powers-2}) decreases linearly with $j$, while $k^{(j)}=k$ remains constant. Thus $k^{(j)}>m^{(j)}$ for large enough $j$, in which case $x^j=\zero$ by (\ref{eq:powers-1}). The index $i(x)$ is the smallest such  $j$ and is given by (\ref{eq:index-of-nilpotent}). The proof for $d<0$ is similar.
\end{proof}

\begin{remark} 
In the special case of integer parameters ($x\in M_n$) we can show that (\ref{eq:index-of-nilpotent}) reduces to formula (28) of \cite{fikioris-fikioris} (which involves the ceiling rather than the floor function). However, (28) of \cite{fikioris-fikioris} \textit{does not hold} for the more general case $x\in\overline{M}_n$.
\end{remark}

\begin{corollary}
\label{th:periodic}
$\overline{M}_n$ and $\overline{S}_n=\overline{M}_n\setminus\{\one\}$ are
periodic inverse semigroups. $\overline{M}_n$ is \textbf{not} categorical at zero, and neither is  $\overline{S}_n$. 
\end{corollary}

\begin{proof}
A semigroup is \textit{periodic} when all its elements are of finite order, i.e., when the monogenic subsemigroup generated by any semigroup element has finite cardinality \cite{Howie}. As all $x\in \overline{M}_n$ are idempotent and/or nilpotent, both $\overline{M}_n$ and $\overline{S}_n$ are periodic.  

By definition \cite{Lawson, bulman}, a semigroup with zero is categorical at zero if $xyz=\zero$ implies $xy=\zero$ or $yz=\zero$. This is not true of $\overline{M}_n$ or $\overline{S}_n$, because there are nilpotents $x$ of index $i(x)=3$.
\end{proof}

\section{Further results}
\label{section:additional}

In this section, we determine Green's relations and show that $\overline{M}_n$ is a strongly $E^*$-unitary, combinatorial, fundamental, and completely semisimple inverse monoid. Then, we explicitly determine all ideals of $\overline{M}_n$, discuss issues pertaining to $j$th roots, and point out differences between $\overline{M}_n$ and $M_n$. Finally, we show that $\overline{M}_n$ is a supersemigroup of the Brandt semigroup, and prove that $\overline{M}_n$ has infinite Sierpi\'nski rank. 

We have seen (Corollary~\ref{corollary:e*unitary}) that $\overline{M}_n$ is a $E^*$-unitary inverse semigroup. We now demonstrate that $\overline{M}_n$ belongs to the narrower class of strongly $E^*$-unitary inverse semigroups \cite{bulman,lawson-e*,mcalister}. In what follows, $G$ denotes the (multiplicative) circle group, by which we mean the complex numbers on the unit circle, 
\begin{equation}
    G=\{z\in\mathbb{C}: \abs{z}=1\}=\{e^{i\theta}\in\mathbb{C}: -\pi<\theta\le\pi\}.
\end{equation}
The unique idempotent of $G$ is $1=e^{i0}$. $G^0=G\cup \{0\}$ is a group with zero.
\begin{theorem}
$\overline{M}_n$ is a strongly $E^*$-unitary inverse semigroup. 
\end{theorem}
    \begin{proof}
        By definition---and as explained in \cite{bulman,mcalister,lawson-e*}---an inverse semigroup is strongly $E^*$-unitary if there exists an idempotent-pure $0$-morphism $\varphi$ from the semigroup into a group with zero. The proof that follows is constructive. 
        
        By (\ref{eq:restrictions-further}), the map $\varphi: \overline{M}_n\rightarrow G^0$ given by 
\begin{equation*}
\label{eq:map-for-strongly}
    \zero\mapsto 0, \quad x=\prec k, d, m \succ\,\mapsto \exp\left({\frac{id}{n-1}}\right),
\end{equation*}
maps $\overline{M}_n\setminus\{\zero\}$ onto a portion of $G$, namely the arc $\{e^{i\theta}\in\mathbb{C}: \abs{\theta}\le 1\}$. $\varphi$~is idempotent-pure because only idempotents ($d=0$) map to $1$, and $0$-restricted because only $\zero$ maps to the origin $0$. Now suppose that $x,y\in \overline{M}_n$ with $xy\ne 0$, and set $x=\prec k, d, m \succ$ and $y=\prec k', d', m' \succ$. (\ref{eq:condition-for-nonzero}) gives $\abs{d+d'}\le n-1$, so that the image of the nonzero product $xy$ belongs to the aforementioned arc. Furthermore,
\begin{equation*}
    \varphi(xy)=\exp\left[{\frac{i(d+d')}{n-1}}\right]=\exp\left({\frac{id}{n-1}}\right)\exp\left({\frac{id'}{n-1}}\right)=\varphi(x)\varphi(y),
\end{equation*}
meaning that $\varphi$ is a $0$-morphism.
    \end{proof}

The theorem below gives Green's relations on $\overline{M}_n$, which turn out to be very similar to those in $M_n$ (see Theorem~12 of \cite{fikioris-fikioris}, but take into account Remark~\ref{remark:h-instead-of-rnk}). Our derivations of  $\mathcal{R}$, $\mathcal{L}$, $\mathcal{H}$, $\mathcal{D}$ use $\mathbf{r}$ and $\mathbf{d}$---and stress graphical interpretations. Our theorem further shows that $\mathcal{J}=\mathcal{D}$; while this is also true in $M_n$, it requires a different proof because $\overline{M}_n$ is not finite.  
\edit{
\begin{theorem}
\label{th:green}
    In the inverse monoid $\overline{M}_n$, Green's relations for any two nonzero elements $x=\prec k, d, m \succ$ and $y=\prec k', d', m' \succ$ are as follows.
    \begin{equation}
    \label{eq:green-R}
        x\mathcal{R}y
        \iff  
        \mathbf{r}(x)=\mathbf{r}(y)\iff
        k=k'\ \mathrm{and}\ m=m',
    \end{equation}
    \begin{equation}
    \label{eq:green-L}
       x\mathcal{L}y\iff 
       \mathbf{d}(x)=\mathbf{d}(y)\iff k+d=k'+d'\ \mathrm{and}\ m+d=m'+d'.
    \end{equation}
For $x,y\in \overline{M}_n$,
    \begin{equation}
    \label{eq:green-H}
        x\mathcal{H}y
        \iff\left(
        x\mathcal{R}y\ 
        \mathrm{and}\ 
        x\mathcal{L}y \right)
        \iff\ x=y,
    \end{equation}
    \begin{equation}
    \label{eq:green-D}
        x\mathcal{D}y\iff \left(
        \exists z\in \overline{M}_n: \mathbf{d}(z)=\mathbf{d}(x)\ \mathrm{and} \ \mathbf{r}(z)=\mathbf{r}(y) \right)
        \iff h(x)=h(y),
    \end{equation}
    \begin{equation}
    \label{eq:green-J}
        x\mathcal{J}y\iff h(x)=h(y).
    \end{equation}
    In all cases, $\zero$ forms a class of its own,
    \begin{equation}
    \label{eq:green-zero}
R_\zero=L_\zero=H_\zero=D_\zero=J_\zero=\{\zero\}.
    \end{equation}
\end{theorem}
}
\begin{proof}
In (\ref{eq:green-R})--(\ref{eq:green-D}), the first equivalences---which express Green's relations in terms of  $\mathbf{r}$ and $\mathbf{d}$---are standard results which hold for all inverse semigroups \cite{Lawson,lawson2023introduction}. The second equivalences in (\ref{eq:green-R}) and (\ref{eq:green-L}), as well as the special cases $R_\zero=L_\zero=\{\zero\}$, then follow from (\ref{eq:productxtimesinversex}) and (\ref{eq:productinversextimesx}).  The second equivalence in (\ref{eq:green-H})  is an immediate consequence of (\ref{eq:green-R}), (\ref{eq:green-L}), and $R_\zero=L_\zero=\{\zero\}$. 

We now turn to $\mathcal{D}$. By the discussions in Section~\ref{section:graphical}, the segment $z$ in (\ref{eq:green-D}) is, concurrently, a vertical translation of $x$ and a horizontal translation of $y$; and it is graphically apparent---see especially (\ref{eq:height-maintained})---that such a $z$ exists iff $h(x)=h(y)$. (The paper \cite{fikioris-fikioris} contains an explicit expression for $z$, which remains valid for  $\overline{M}_n$.) 

(\ref{eq:green-J}) is tantamount to $\mathcal{J}=\mathcal{D}$, which we show in two ways: Firstly, it holds by virtue of Corollary~\ref{th:periodic}, 
because $\mathcal{J}=\mathcal{D}$ in any semigroup that is periodic \cite{Howie}. Secondly, we know (Corollary 3.19 of \cite{Lawson}) that $\mathcal{J}=\mathcal{D}$ in any inverse semigroup satisfying
\begin{equation}
\label{eq:completelysemisimple}
    x\mathcal{D} y \ \mathrm{and}\ x\le y\implies x=y.
\end{equation}
By (\ref{eq:green-D}), we must show 
\begin{equation*}
    h(x)=h(y) \ \mathrm{and}\ x\le y\implies x=y.
\end{equation*}
which is apparent graphically, or can be proved using (\ref{eq:natural-partial-order}) and (\ref{eq:v}). 
\end{proof}

\begin{corollary}
    $\overline{M}_n$ is a completely semisimple, combinatorial, and fundamental inverse monoid. 
\end{corollary}
\begin{proof}
    By definition \cite{Lawson}, an inverse semigroup is completely semisimple when (\ref{eq:completelysemisimple}) is satisfied. As $\mathcal{H}$ is the equality relation, $\overline{M}_n$ is a combinatorial semigroup \cite{Lawson}. Finally, all combinatorial inverse semigroups are fundamental \cite{Lawson}.
\end{proof}

 Since $\overline{M}_n$ is fundamental, it is a full inverse submonoid of the Munn monoid on $\mathbf{E}\left(\overline{M}_n\right)$ \cite{Howie,Lawson,lawson2023introduction}. By Corollary~\ref{corollary:semilattice}, any two of the principal ideals of $\mathbf{E}\left(\overline{M}_n\right)$ are isomorphic. In other words \cite{Howie}, $\mathbf{E}\left(\overline{M}_n\right)$ is a uniform semilattice.
 
Before proceeding, we develop two lemmas involving the height function. They will help us obtain the principal ideals
of $\overline{M}_n$, and discuss the subsemigroup $\{\zero\}\cup\overline{P}_n$. 
\begin{lemma} 
\label{lemma:yequalszxw-1}
Let $x,y\in\overline{M}_n$. Then $h(y)\le h(x)$ iff there exist $z,w\in\overline{M}_n$ such that
\begin{equation}
\label{eq:yeqzxw}
    y=zxw. 
\end{equation}
Furthermore, if $0\le h(y)\le h(x)$ with $x=\prec k, d, m \succ$ and $y=\prec k', d', m' \succ$, then (\ref{eq:yeqzxw}) is satisfied by the nonzero elements $z=\prec k_z, d_z, m_z \succ$ and $w=\prec k_w, d_w, m_w \succ$ where
\begin{equation}
\label{eq:zdefinition}
    k_z=k',\quad d_z=k-k',\quad  m_z=m';
    \end{equation}
\begin{equation}
\label{eq:wdefinition}
    k_w=k+d,\quad d_w=k'+d'-k-d,\quad  m_w=k+d+m'-k'.
\end{equation}
\end{lemma}
\begin{proof}
If (\ref{eq:yeqzxw}) holds, then $h(y)\le h(x)$ by (\ref{eq:rank-property}). Conversely, suppose that $h(y)\le h(x)$. If $x=\zero$ or $y=\zero$, (\ref{eq:yeqzxw}) is trivial. We thus take $x,y\in \overline{M}_n\setminus\{\zero\}$; call $x=\prec k, d, m \succ$, $y=\prec k', d', m' \succ$; and define $z$, $w$ by (\ref{eq:zdefinition}), (\ref{eq:wdefinition}).
 By (\ref{eq:v}), the assumption $h(y)\le h(x)$ amounts to
\begin{equation}
\label{eq:vylessthanvx}
m'-k'\le m-k.
\end{equation}
Write the conditions in (\ref{eq:m-definition}) for $d$, $k$, $m$, and again for $d'$, $k'$, $m'$. Upon invoking (\ref{eq:zdefinition})--(\ref{eq:vylessthanvx}), we can easily deduce identical conditions for $d_z$, $k_z$, $m_z$ and for $d_w$, $k_w$, $m_w$. Thus $z$ and $w$ are well-defined elements of $\overline{M}_n\setminus\{\zero\}$. Finally, a quick calculation based on the multiplication formula (\ref{eq:multiplication-finite}) verifies (\ref{eq:yeqzxw}).
\end{proof}
\begin{lemma}
\label{lemma:yequalszxw-2}
    Let $y\in \{\zero\}\cup\overline{P}_n$. Let $x\in\overline{M}_n\setminus\{\zero\}$. Then there exist $z,w\in \{\zero\}\cup\overline{P}_n$ such that $y=zxw$.
\end{lemma}
\begin{proof}
If $y=\zero$, the statement is trivial. Otherwise $y\in\overline{P}_n$, so $0=h(y)\le h(x)$ by  (\ref{eq:rank-special-cases}) and (\ref{eq:rank-property}). Thus (\ref{eq:yeqzxw}) holds, where $z,w\in \overline{M}_n$ are given by (\ref{eq:zdefinition}) and (\ref{eq:wdefinition}) with $m'=k'$. It follows that $k_z=m_z$ and $k_w=m_w$, so that $z,w\in \overline{P}_n$. 
\end{proof}

The result (Theorem~\ref{th:green}) that $J_x$ consists of all segments of height $h(x)$ means that elements whose heights are equal generate the same principal ideal. We now go beyond this observation and explicitly describe all ideals, whether principal or not. It will be seen that, as opposed to $M_n$ (see Theorem~13 of \cite{fikioris-fikioris}), $\overline{M}_n$ has (two-sided) ideals that are not principal. In the theorem that follows, these non-principal ideals are denoted by $K_\mu$.

\begin{theorem}
\label{th:principal-ideals} 
The principal ideals of $\overline{M}_n$ are precisely the following sets ${I}_\mu$, 
\begin{equation}
\label{eq:principal-ideals-definition}
    {I}_\mu=\{y\in \overline{M}_n: h(y)\le \mu\},\quad \mu\in \{-1\}\cup[0,n-1].
\end{equation}
In particular, 
\begin{equation}
\label{eq:ideals-special-cases}
    {I}_{-1}=\{\zero\},\quad {I}_{0}=\{\zero\}\cup \overline{P}_n,\quad {I}_{n-1}=\overline{M}_n.
\end{equation}
The ${I}_\mu$ defined in (\ref{eq:principal-ideals-definition}) are also given by
\begin{equation}
\label{eq:principal-ideals-given-by}
    {I}_\mu=\overline{M}_n x\overline{M}_n,
\end{equation}
in which $x$ is any element of $\overline{M}_n$ with $h(x)= \mu$. 

The non-principal ideals of $\overline{M}_n$ are precisely the following sets
${K}_\mu$,
\begin{equation}
\label{eq:ideals-definition}
    {K}_\mu=\{y\in \overline{M}_n: h(y)< \mu\},\quad \mu\in (0,n-1].
\end{equation}

It follows that the collections $\{{I}_\mu\}$ and  $\{{K}_\mu\}$ are both strictly totally ordered; that is,
${I}_\mu\subset {I}_\xi$ and ${K}_\mu\subset {K}_\xi$ whenever $\mu<\xi$.
\end{theorem}

\begin{proof} Define the sets $I_\mu$ by (\ref{eq:principal-ideals-definition}) and choose an $x\in\overline{M}_n$ such that $h(x)=\mu$. The iff statement of Lemma~\ref{lemma:yequalszxw-1} can then be rephrased as: $y\in I_\mu\iff y\in \overline{M}_n x \overline{M}_n$. We have thus shown (\ref{eq:principal-ideals-given-by}). Therefore all principal ideals are given in (\ref{eq:principal-ideals-definition}).

The special cases in (\ref{eq:ideals-special-cases}) follow from (\ref{eq:rank-special-cases}) and (\ref{eq:principal-ideals-definition}). 

We now let ${I}$ be an \textit{arbitrary} ideal. From
\begin{equation}
\label{eq:ideals-union}
    {I}=\overline{M}_n{I}\overline{M}_n=\cup_{x\in {I}}\overline{M}_n x\overline{M}_n,
\end{equation}
 we see that ${I}$ is a union of principal ideals. By (\ref{eq:principal-ideals-definition}), these are totally ordered sets. If ${I}$ contains an element $x$ such that $h(y)\le h(x)$ for all $y\in{I}$, then the union in  (\ref{eq:ideals-union}) equals ${I}_\mu$, where $\mu=h(x)=\mathrm{max}_{y\in I}\{h(y)\}$, so that  ${I}$ is itself a principal ideal. If there is no such element $x\in I$---i.e., if the subset $\{h(y): y\in{I}\}$ of $\mathbb{R}$
has no maximum---then the union in (\ref{eq:ideals-union}) is one of the totally ordered sets in (\ref{eq:ideals-definition}), namely ${K}_\mu$, where $\mu=\mathrm{sup}_{y\in I}\{h(y)\}$.

It remains to show, conversely, that all the ${K}_\mu$ defined in (\ref{eq:ideals-definition}) are ideals. Let $y\in\overline{M}_n{K}_\mu$, so that $y=zw$ with $z\in \overline{M}_n$ and $w\in {K}_\mu$. It follows from (\ref{eq:rank-property}) that $h(y)\le h(w)$. Since $h(w)<\mu$, we have $h(y)<\mu$, so that $y\in {K}_\mu$. Hence $\overline{M}_n{K}_\mu
\subseteq {K}_\mu$, so ${K}_\mu$ is a left ideal by definition. Similarly, ${K}_\mu$ is a right ideal. Thus ${K}_\mu$ is a two-sided ideal, completing our proof. 
\end{proof}

Theorem~6 of \cite{fikioris-fikioris} discusses $j$th roots for the integer case: In $M_n$, a nonzero element $x=\prec k, d, m \succ$ has a $j$th root iff $d$ is an integer multiple of $j$; and the $j$th root, when it exists, is unique. The theorem that follows shows that, in  $\overline{M}_n$, a unique root $y$ \textit{always} exists. 
In other words (and in complete analogy to the case of $\mathbb{R}_{>0}$ and its subset $\mathbb{N}$) any nonzero element $x\in \overline{M}_n$ ($x\in \mathbb{R}_{>0}$) has a unique root $y\in\overline{M}_n$ ($y\in \mathbb{R}_{>0}$); but in the special case $x\in M_n$ ($x\in \mathbb{N}$), the said root $y$ is not necessarily in $M_n$ (in $\mathbb{N}$).

\begin{theorem}
\label{th:roots}
Let $j\in\mathbb{N}$. The element $x=\prec k, d, m \succ\,\,\in \overline{M}_n\setminus\{\zero\}$ has a unique $j$th root in $\overline{M}_n$. It is given by $y=\prec k', d', m' \succ\,\,\in \overline{M}_n\setminus\{\zero\}$, where
\begin{equation}
\label{eq:root-parameters}
k'=k+(j-1)\min(0,d'),\quad
d'=\frac{d}{j},\quad  m'=m+(j-1)\max(0,d').
\end{equation}
\end{theorem}
\begin{proof}
Assume $d\ge 0$, so that (\ref{eq:m-definition}) implies
\begin{equation}
\label{eq:root-condition-on-x}
    1\le k\le m\le n-d.
\end{equation}
We seek $y\in \overline{M}_n$ such that $x=y^j$. As $y\ne \zero$, we set $y=\prec k', d', m' \succ$. By Lemma~\ref{lemma:powers}, $d'=d/j\ge 0$. Invoking (\ref{eq:m-definition}), we thus require
\begin{equation}
\label{eq:root-condition-2}
    1\le k'\le m'\le n-d'.
\end{equation}
By Lemma~\ref{lemma:powers}, $x=y^j$ is equivalent to the three equations
\begin{equation*}
\label{eq:root-condition-1}
k=k',\quad d=jd',\quad  m=m'-(j-1)d'.
\end{equation*}
These are uniquely solvable for $k'$, $d'$,  $m'$ and the solution is given in (\ref{eq:root-parameters}). Eqns. (\ref{eq:root-parameters}) and  (\ref{eq:root-condition-on-x}) then imply (\ref{eq:root-condition-2}), completing the proof for $d\ge 0$. We can extend to $d<0$ by taking the inverse.
\end{proof}

\begin{remark}
    One could also consider the submonoid $A_n$ of $\overline{M}_n$ in which $k,d,m\in\mathbb{Q}$. For $x\in A_n\setminus\{\zero\}$, the unique $j$th root $y$ given in (\ref{eq:root-parameters}) also belongs to $A_n\setminus\{\zero\}$. Thus in  $A_n\setminus\{\zero\}$, a unique root $y$ always exists. Consequently, despite the aforementioned analogy of $\overline{M}_n$ to  $\mathbb{R}_{>0}$ and $M_n$ to $\mathbb{N}$, the submonoid $A_n$ is not analogous to $\mathbb{Q}_{>0}$.
\end{remark}

By (\ref{eq:m-definition}) and (\ref{eq:points}), the set $\overline{B}_n=\{\zero\}\cup\overline{P}_n$ is given by 
\begin{equation}
\label{eq:brandt-definition}
\overline{B}_n=\{\zero\}\cup\overline{P}_n=\{\zero\}
\cup
\{\prec k, d, k \succ: \  1-\min(0,d)\le k \le n-\max(0,d)\}.
\end{equation}

Example~2 of \cite{fikioris-fikioris} shows that, in the integer case, the subsemigroup $B_n$ of $\overline{B}_n$ is isomorphic to a certain Brandt semigroup of finite cardinality. The theorem that follows is a generalization that can be proved in a number of ways. We give a proof that builds upon previous results in the present paper, as well as concepts and results on inverse semigroups that can be found in \cite{Lawson}. 

\begin{theorem}
$\overline{B}_n$ is a Brandt semigroup.
\end{theorem}
\begin{proof}
By (\ref{eq:multiplication-finite}), (\ref{eq:transpose}), and (\ref{eq:brandt-definition}), $\overline{B}_n$ is an inverse subsemigroup of $\overline{M}_n$. Therefore $\overline{B}_n$ inherits its natural partial order $\le$ from $\overline{M}_n$. By (\ref{eq:brandt-definition}) and  Corollary~\ref{corollary:natural-partial-order}, $x\le y$ iff $x=y$ ($x,y\in \overline{P}_n$), meaning  that in $\overline{P}_n=\overline{B}_n\setminus\{\zero\}$, the $\le$  reduces to an equality.  Equivalently \cite{Lawson}, all idempotents of $\overline{B}_n\setminus\{\zero\}$ are primitive.

Now let $I\subseteq \overline{B}_n$ be an ideal of $\overline{B}_n$.  Assume $I\ne \{\zero\}$, so that some nonzero $x$ belongs to $I$.  Choose any $y$ in $\overline{B}_n$. By Lemma~\ref{lemma:yequalszxw-2}, this $y$ 
 belongs to the principal ideal  $\overline{B}_nx\overline{B}_n$, so that $\overline{B}_n\subseteq\overline{B}_nx\overline{B}_n$. As $\overline{B}_nx\overline{B}_n\subseteq I$, we further have $\overline{B}_n\subseteq I$, so $I=\overline{B}_n$. Therefore the only ideals of $\overline{B}_n$ are  $\{\zero\}$ and $\overline{B}_n$ itself, meaning that $\overline{B}_n$ is $0$-simple.

Inverse, $0$-simple semigroups with at least one primitive idempotent are Brandt semigroups \cite{Lawson}, completing our proof. 
\end{proof}

Corollary~6 of \cite{fikioris-fikioris} determines a minimal generating set for $M_n$ that, for any $n$, consists of only three elements.  Thus the rank of $M_n$ (integer case) is 3. Since $\overline{M}_n$ is uncountable, the situation is very different. In what follows, we prove that $\overline M_n$ has infinite Sierpi\'{n}ski rank \cite{peresse2006,peresse2009generating,east2012}, meaning that there are countable subsets of $\overline M_n$ that cannot be generated by finitely many elements of $\overline M_n$. 
\begin{theorem}
The Sierpi\'{n}ski rank of $\overline M_n$ is infinite.
\end{theorem}

\begin{proof}
It suffices to prove that the Sierpi\'{n}ski rank of $\overline{S}_n=\overline{M}_n\setminus\{\one\}$ is infinite, see  (\ref{eq:sn-semigroup-2}). By (\ref{eq:m-definition}),  the countable set  $A_n = \{y_i : i \in \mathbb{N}\}$ with elements
\begin{equation*}
y_i=\prec 1, 2^{-i}, n-2^{-i} \succ,
\end{equation*}
is a well-defined subset of $\overline{S}_n$. By (\ref{eq:v}), the sequence of heights $h(y_i)$ increases, with 
\begin{equation}
\label{eq:sierpinski-nminus1}
\sup_{i\in\mathbb{N}}h(y_i)=\lim_{i\to\infty} \left(n-2^{-i}-1\right)=n-1. 
\end{equation}

Assume that $A_n$ is generated by a finite set with $r$ elements
$G_n = \{g_1, \ldots, g_r\}$. For every $i \in \mathbb{N}$ this implies that $y_i = g_{i_1}g_{i_2}\ldots g_{i_s}$ for some $i_1, i_2,\ldots, i_s \in \{1,\ldots, r\}$. By \eqref{eq:rank-property} this means that $h(y_i) \le \min\{h(g_{i_1}),h(g_{i_2}),\ldots, h(g_{i_s})\} \le h_{\max}$, where $h_{\max} = \max\{h(g_j) : j \in \{1,\ldots, r\}\}$. Since $\one\notin G_n\subset\overline{S}_n$, (\ref{eq:rank-range}) and (\ref{eq:rank-special-cases}) give $h_{\max}<n-1$, which contradicts (\ref{eq:sierpinski-nminus1}).
\end{proof}

\bibliography{ref.bib}


\begin{thebibliography}{12}
\ifx \bisbn   \undefined \def \bisbn  #1{ISBN #1}\fi
\ifx \binits  \undefined \def \binits#1{#1}\fi
\ifx \bauthor  \undefined \def \bauthor#1{#1}\fi
\ifx \batitle  \undefined \def \batitle#1{#1}\fi
\ifx \bjtitle  \undefined \def \bjtitle#1{#1}\fi
\ifx \bvolume  \undefined \def \bvolume#1{\textbf{#1}}\fi
\ifx \byear  \undefined \def \byear#1{#1}\fi
\ifx \bissue  \undefined \def \bissue#1{#1}\fi
\ifx \bfpage  \undefined \def \bfpage#1{#1}\fi
\ifx \blpage  \undefined \def \blpage #1{#1}\fi
\ifx \burl  \undefined \def \burl#1{\textsf{#1}}\fi
\ifx \doiurl  \undefined \def \doiurl#1{\url{https://doi.org/#1}}\fi
\ifx \betal  \undefined \def \betal{\textit{et al.}}\fi
\ifx \binstitute  \undefined \def \binstitute#1{#1}\fi
\ifx \binstitutionaled  \undefined \def \binstitutionaled#1{#1}\fi
\ifx \bctitle  \undefined \def \bctitle#1{#1}\fi
\ifx \beditor  \undefined \def \beditor#1{#1}\fi
\ifx \bpublisher  \undefined \def \bpublisher#1{#1}\fi
\ifx \bbtitle  \undefined \def \bbtitle#1{#1}\fi
\ifx \bedition  \undefined \def \bedition#1{#1}\fi
\ifx \bseriesno  \undefined \def \bseriesno#1{#1}\fi
\ifx \blocation  \undefined \def \blocation#1{#1}\fi
\ifx \bsertitle  \undefined \def \bsertitle#1{#1}\fi
\ifx \bsnm \undefined \def \bsnm#1{#1}\fi
\ifx \bsuffix \undefined \def \bsuffix#1{#1}\fi
\ifx \bparticle \undefined \def \bparticle#1{#1}\fi
\ifx \barticle \undefined \def \barticle#1{#1}\fi
\bibcommenthead
\ifx \bconfdate \undefined \def \bconfdate #1{#1}\fi
\ifx \botherref \undefined \def \botherref #1{#1}\fi
\ifx \url \undefined \def \url#1{\textsf{#1}}\fi
\ifx \bchapter \undefined \def \bchapter#1{#1}\fi
\ifx \bbook \undefined \def \bbook#1{#1}\fi
\ifx \bcomment \undefined \def \bcomment#1{#1}\fi
\ifx \oauthor \undefined \def \oauthor#1{#1}\fi
\ifx \citeauthoryear \undefined \def \citeauthoryear#1{#1}\fi
\ifx \endbibitem  \undefined \def \endbibitem {}\fi
\ifx \bconflocation  \undefined \def \bconflocation#1{#1}\fi
\ifx \arxivurl  \undefined \def \arxivurl#1{\textsf{#1}}\fi
\csname PreBibitemsHook\endcsname

%%% 1
\bibitem{Ganyushkin}
\begin{bbook}
\bauthor{\bsnm{Ganyushkin}, \binits{O.}},
\bauthor{\bsnm{Mazorchuk}, \binits{V.}}:
\bbtitle{Classical Finite Transformation Semigroups: an Introduction}.
\bpublisher{Springer},
\blocation{London, UK}
(\byear{2008})
\end{bbook}
\endbibitem

%%% 2
\bibitem{solomon}
\begin{barticle}
\bauthor{\bsnm{Solomon}, \binits{L.}}:
\batitle{Representations of the rook monoid}.
\bjtitle{J. of Algebra}
\bvolume{256},
\bfpage{309}--\blpage{342}
(\byear{2002})
\end{barticle}
\endbibitem

%%% 3
\bibitem{fikioris-fikioris}
\begin{barticle}
\bauthor{\bsnm{Fikioris}, \binits{G.}},
\bauthor{\bsnm{Fikioris}, \binits{G.}}:
\batitle{A subsemigroup of the rook monoid}.
\bjtitle{Semigroup Forum}
\bvolume{105}(\bissue{1}),
\bfpage{191}--\blpage{216}
(\byear{2022}).
\bcomment{Springer-Verlag New York.}
\end{barticle}
\endbibitem

%%% 4
\bibitem{Lawson}
\begin{bbook}
\bauthor{\bsnm{Lawson}, \binits{M.V.}}:
\bbtitle{Inverse Semigroups: The Theory of Partial Symmetries}.
\bpublisher{World Scientific},
\blocation{Singapore}
(\byear{1998})
\end{bbook}
\endbibitem

%%% 5
\bibitem{lawson2023introduction}
\begin{botherref}
\oauthor{\bsnm{Lawson}, \binits{M.V.}}:
Introduction to inverse semigroups.
arXiv preprint (to appear as book chapter) arXiv:2304.13580
(2023)
\end{botherref}
\endbibitem

%%% 6
\bibitem{Howie}
\begin{bbook}
\bauthor{\bsnm{Howie}, \binits{J.M.}}:
\bbtitle{Fundamentals of Semigroup Theory}.
\bpublisher{Oxford University Press},
\blocation{Oxford, UK}
(\byear{1995})
\end{bbook}
\endbibitem

%%% 7
\bibitem{bulman}
\begin{barticle}
\bauthor{\bsnm{Bulman-Fleming}, \binits{S.}},
\bauthor{\bsnm{Fountain}, \binits{J.}},
\bauthor{\bsnm{Gould}, \binits{V.}}:
\batitle{Inverse semigroups with zero: covers and their structure}.
\bjtitle{Journal of the Australian Mathematical Society}
\bvolume{67}(\bissue{1}),
\bfpage{15}--\blpage{30}
(\byear{1999})
\end{barticle}
\endbibitem

%%% 8
\bibitem{lawson-e*}
\begin{bbook}
\bauthor{\bsnm{Lawson}, \binits{M.V.}}:
\bbtitle{E*-unitary inverse semigroups},
pp. \bfpage{195}--\blpage{214}
(\byear{2002})
\end{bbook}
\endbibitem

%%% 9
\bibitem{mcalister}
\begin{bbook}
\bauthor{\bsnm{McAlister}, \binits{D.B.}}:
\bbtitle{An introduction to {E}*-unitary inverse semigroups—from an old
  fashioned perspective},
pp. \bfpage{133}--\blpage{150}
(\byear{2004})
\end{bbook}
\endbibitem

%%% 10
\bibitem{peresse2006}
\begin{barticle}
\bauthor{\bsnm{Mitchell}, \binits{J.D.}},
\bauthor{\bsnm{P{\'e}resse}, \binits{Y.}},
\bauthor{\bsnm{Quick}, \binits{M.R.}}:
\batitle{Generating sequences of functions}.
\bjtitle{Quarterly J. Mathematics}
\bvolume{58},
\bfpage{71}--\blpage{79}
(\byear{2007})
\end{barticle}
\endbibitem

%%% 11
\bibitem{peresse2009generating}
\begin{botherref}
\oauthor{\bsnm{P{\'e}resse}, \binits{Y.}}:
Generating uncountable transformation semigroups.
PhD thesis,
University of St Andrews
(2009)
\end{botherref}
\endbibitem

%%% 12
\bibitem{east2012}
\begin{barticle}
\bauthor{\bsnm{East}, \binits{J.}}:
\batitle{Generation of infinite factorizable inverse monoids}.
\bjtitle{Semigroup Forum}
\bvolume{84},
\bfpage{267}--\blpage{283}
(\byear{2012})
\end{barticle}
\endbibitem

\end{thebibliography}

\end{document}